\documentclass[onefignum,onetabnum]{siamonline190516}

\usepackage[utf8]{inputenc}
\usepackage[left=2cm,right=2cm]{geometry}
\usepackage{subcaption}
\DeclareMathOperator*{\argmin}{arg\,min}
\usepackage{booktabs}

\usepackage{todonotes}

\usepackage{lipsum}
\usepackage{amsfonts}
\usepackage{graphicx}
\usepackage{epstopdf}
\usepackage{algorithmic}
\ifpdf
  \DeclareGraphicsExtensions{.eps,.pdf,.png,.jpg}
\else
  \DeclareGraphicsExtensions{.eps}
\fi

\usepackage{enumitem}
\setlist[enumerate]{leftmargin=.5in}
\setlist[itemize]{leftmargin=.5in}

\newsiamremark{remark}{Remark}
\newsiamremark{hypothesis}{Hypothesis}
\crefname{hypothesis}{Hypothesis}{Hypotheses}
\newsiamthm{claim}{Claim}

\headers{Follow the bisector: a simple method for multi-objective optimization}{A. Katrutsa, D. Merkulov, N. Tursynbek and I. Oseledets}

\title{Follow the bisector: a simple method for multi-objective optimization\footnotemark[1] 
\footnotetext[1]{Submitted to the editors July 13, 2020.
\funding{This study was supported by the Ministry of Education and Science of the Russian
Federation grant 14.756.31.0001}}}

\author{Alexandr Katrutsa\footnotemark[2]
\and Daniil Merkulov\footnotemark[2] \and Nurislam Tursynbek\footnotemark[2]~\footnotemark[3]~\footnotemark[4] \and Ivan Oseledets\footnotemark[2]
\footnotetext[2]{Skolkovo Institute of Science and Technology 
  (\email{aleksandr.katrutsa@phystech.edu}, \email{daniil.merkulov@skolkovotech.ru}, \email{nurislam.tursynbek@skoltech.ru}, \email{i.oseledets@skoltech.ru}).}
 \footnotetext[3]{Higher School of Economics}
\footnotetext[4]{Huawei Moscow Research Center}
}

\usepackage{amsopn}

\begin{document}
\maketitle

\begin{abstract}
    This study presents a novel Equiangular Direction Method (EDM) to solve a multi-objective optimization problem.
    We consider optimization problems, where multiple differentiable losses have to be minimized.
    The presented method computes descent direction in every iteration to guarantee equal relative decrease of objective functions.
    This descent direction is based on the normalized gradients of the individual losses. 
    Therefore, it is appropriate to solve multi-objective optimization problems with multi-scale losses.
    We test the proposed method on the imbalanced classification problem and multi-task learning problem, where standard datasets are used. EDM is compared with other methods to solve these problems. 
\end{abstract}



\section{Introduction}
Many problems in machine learning and deep learning require the minimization of different loss functions simultaneously. 
Such problems appear, for example, in multi-task learning~\cite{zhang2017survey}, where the single model is trained to be good at different (possibly conflicting) tasks.
The survey of machine learning problem statements including multiple objective functions is presented in~\cite{jin2006multi}.
In particular, an imbalanced classification problem can be solved efficiently by introducing several objective functions~\cite{soda2011multi}.
Also, the feature selection problem is naturally multi-objective since the selected subset of features has to be non-redundant and gives accurate and stable trained model, simultaneously~\cite{xue2012particle,katrutsa2015stress}. 
The exploiting of a multi-objective optimization approach in reinforcement learning problem is discussed in~\cite{liu2014multiobjective,van2014multi}. 

Denote by $\mathcal{L}_1(\theta), \ldots, \mathcal{L}_T(\theta)$ $T$ loss functions in multi-objective optimization problem.
These losses depend on the same vector $\theta \in \mathbb{R}^d$.
We consider the case of differentiable loss functions $\mathcal{L}_1,\ldots,\mathcal{L}_T$.
Since the considered losses can be conflicting, i.e. a minimizer of one loss is not a minimizer of another loss, the solution of the multi-objective minimization problem is defined in the following way.
\begin{definition}[Pareto-optimal solution]
\begin{enumerate}
\item A point $\theta'$ dominates $\theta$ for a multi-objective optimization problem, if $\mathcal{L}_i(\theta') \leq \mathcal{L}_i(\theta)$ for $i=1,\ldots,T$ and at least one inequality is strict.
\item A point $\theta^*$ is called Pareto-optimal solution, if there is no other point $\theta'$ that dominates it.
\end{enumerate}
\end{definition}
One of the standard approaches to find Pareto-optimal solution is  scalarization~\cite{ghane2015new}, e.g. weighting sum method that minimizes the conical combination of the losses:
\begin{equation}
\label{harmgrad:sumloss}
 \sum_{i=1}^T w_i \mathcal{L}_i(\theta) \to \min_{\theta}
\end{equation}
where $w_i > 0$ are static or dynamically updated weights of each loss.
The survey on the scalarization approach to solving multi-objective optimization problems is given in~\cite{miettinen2002scalarizing}.
Note that the weighting sum method~\eqref{harmgrad:sumloss} gives different Pareto-optimal points for different weights~$w_i$. 
The drawbacks of this approach to identify Pareto-optimal points are discussed in~\cite{das1997closer}.

\paragraph{Multiple gradient descent algorithm}
Another approach to finding the Pareto-optimal point is multiple-gradient descent algorithm (MGDA)~\cite{desideri2012multiple}. 
The main idea of this algorithm is to find in every iteration descent direction for the considered losses $\mathcal{L}_1,\ldots,\mathcal{L}_T$.
The procedure to find this direction is based on the necessary condition for the point to be a Pareto-optimal solution.
This condition is called Pareto stationarity.
\begin{definition}[Pareto stationarity]
\label{harmgrad:pardef}
Given $T$ differentiable losses $\mathcal{L}_i(\theta), \; i = 1, \ldots, T, \quad \theta \in \mathbb{R}^d$, the point $\hat{\theta}$ is called Pareto-stationary, iff there exists a convex combination of gradients $g_i = \nabla \mathcal{L}_i(\hat{\theta})$ equals to zero, i.e.
\begin{equation}
\label{harmgrad:par-stationary}
\sum_{i=1}^T \alpha_i g_i = 0, \quad \text{where } \alpha_i \geq 0, \; i=1,\ldots,T,  \quad \sum_{i=1}^T \alpha_i = 1.
\end{equation}
\end{definition}
Based on this necessary condition, the descent direction in MGDA is a convex combination of gradients $g_i = \nabla\mathcal{L}_i(\theta)$ whose norm is minimal.
Thus, the following optimization problem has to be solved
\begin{equation}
    \label{harmgrad:min_norm_problem}
        \alpha^* = \argmin_{\alpha \in \Delta_T} \left\| \sum_{i=1}^T \alpha_i g_i \right\|_2^2,
\end{equation}
where $\Delta_T = \left\{ x \in \mathbb{R}^T_+ \mid \sum\limits_{i=1}^T x_i = 1 \right\}$.
Denote by $d_h$ the resulting convex combination of $g_1,\ldots, g_T$, i.e. $d_h = \sum\limits_{i=1}^T\alpha^*_i g_i$.
If $\|d_h\|_2 = 0$, then $\theta$ is a Pareto-stationary point.
Otherwise, direction $-d_h$ is a descent direction for the losses $\mathcal{L}_1, \ldots, \mathcal{L}_T$~\cite{desideri2012multiple}.
Therefore, MGDA updates vector $\theta_k$, where $k$ is the iteration number, similar to the gradient descent for a single-objective optimization problem:
\begin{equation}
    \theta_{k+1} = \theta_k - s_k d^{(k)}_h,
    \label{harmgrad::eq::theta_update}
\end{equation}
where $s_k > 0$ is a learning rate.
Figure~\ref{harmgrad::fig::wh_wb} illustrates the position of the direction $d_h$ in the case of two gradients $g_1$ and $g_2$.
\begin{figure}[!ht]
    \centering
    \begin{subfigure}[b]{0.45\textwidth}
    \centering
    \includegraphics[width=\textwidth]{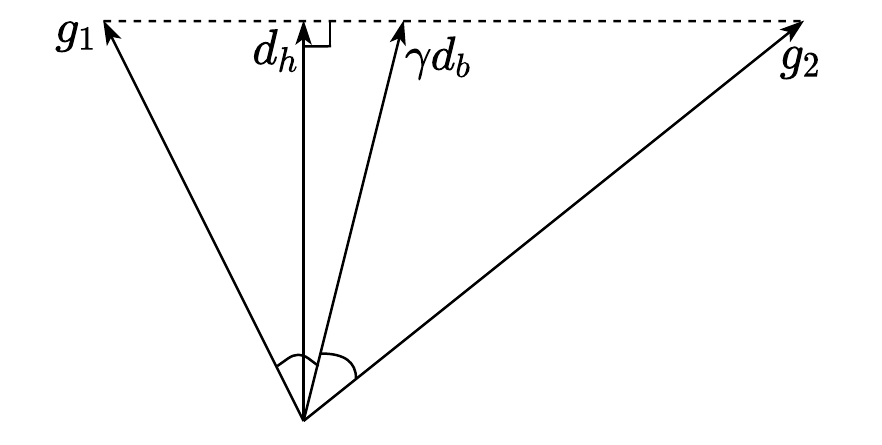}
    \caption{}
    \label{harmgrad::fig::balanced_wh_wb}
    \end{subfigure}
    ~
    \begin{subfigure}[b]{0.45\textwidth}
    \centering
    \includegraphics[width=\textwidth]{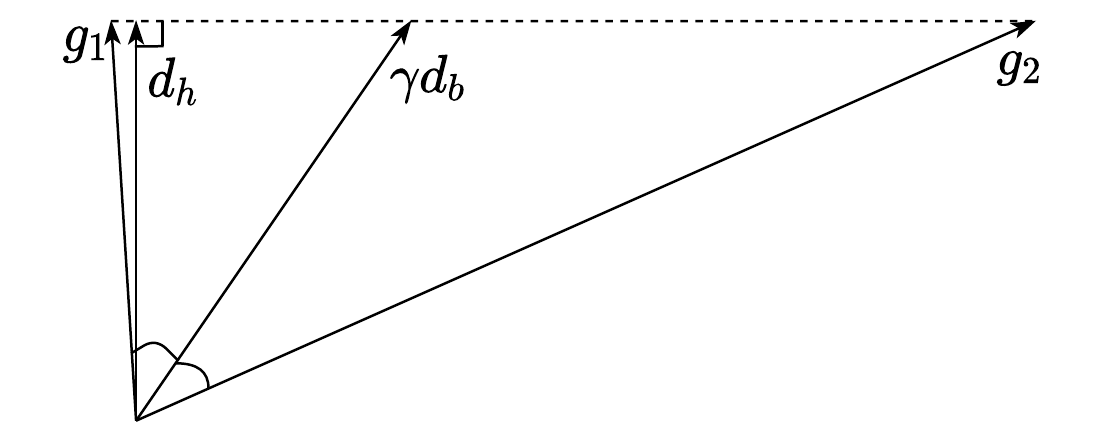}
    \caption{}
    \label{harmgrad::fig::unbalanced_wh_wb}
    \end{subfigure}
    \caption{Geometric interpretation of the directions $d_h$ and $d_b$ in the case of two losses.
    The case of $\|g_1\|_2 \sim \|g_2\|_2$ is shown on the left and the case of $\|g_2\|_2 \gg \|g_1\|_2$ is shown on the right. Normalization factor $\gamma > 0$ scales the vector $d_b$ such that normalized vector $\gamma d_b$ lies in the convex hull of gradients $g_1$ and $g_2$.}
    \label{harmgrad::fig::wh_wb}
\end{figure}
This choice of the descent direction leads to the convergence to the Pareto-stationary point, such that it is the closest to the initial point $\theta_0$ among all other Pareto-stationary points.
Thus, in contrast to the weighting sum method~\eqref{harmgrad:sumloss}, MGDA converges to a particular Pareto-stationary point.  

One more property of the direction $d_h$ is the following (see Theorem 2.2. in~\cite{desideri2012multiple}): if this direction belongs to the interior of the convex hull of the gradients $g_1,\ldots, g_T$, it satisfies
\begin{equation}
\label{harmgrad:equal}
    (d_h, g_i) = \|d_h\|_2^2, \quad i = 1, \ldots, T.
\end{equation}
It means that each loss $\mathcal{L}_i$ gets approximately the same absolute decrease after updating $\theta_k$ according to~\eqref{harmgrad::eq::theta_update} since
\[
\mathcal{L}_i(\theta_k - s_k d^{(k)}_h) \approx \mathcal{L}_i(\theta_k) - s_k(d^{(k)}_h, g_i), \quad i = 1,\ldots, T,
\]
where $d^{(k)}_h$ is the direction $d_h$ in the $k$-th iteration.

This is the starting point for our method: if the individual losses are not balanced (e.g. lets multiply one loss by a large factor), the Pareto front does not change, however the solution of~\eqref{harmgrad:min_norm_problem} changes significantly, compare $d_h$ in Figures~\ref{harmgrad::fig::balanced_wh_wb} and~\ref{harmgrad::fig::unbalanced_wh_wb}.
Instead, we propose to look for the direction $d_b$ that gives the same \emph{relative decrease}:
\begin{equation}
\label{harmgrad:bisector}
    (d_b, g_i) = \|d_b\|_2^2 \Vert g_i \Vert_2,
\end{equation}
or in other words, the direction $d_b$ has the same angle to the gradients $g_1,\ldots, g_T$ (see Theorem~\ref{harmgrad::th::edm_claim}).
Therefore, the proposed method is called \emph{Equiangular Direction Method (EDM)}.
For two gradients $g_1, g_2$, the direction $d_b$ is a bisector of the angle between $g_1, g_2$, see Figure~\ref{harmgrad::fig::wh_wb}.




\section{Equiangular Direction Method (EDM)}
The key step of the EDM is to find the direction $d_b$ that has the same angle with individual gradients $g_i= \nabla \mathcal{L}_i(\theta), i = 1, \ldots, T$.
This direction can be naturally written as a convex combination of the normalized gradients:
\begin{equation}
    d_b = \sum_{i=1}^T \beta_i^* \frac{g_i}{\|g_i\|_2},
    \label{harmgrad:eq:d_b}
\end{equation}
where the coefficients $\beta_i^*$ are the solution of the following convex optimization problem:
\begin{equation}
\label{harmgrad:normgrad}
\beta^* = \argmin_{\beta \in \Delta_T} \left\|  \sum_{i=1}^T \beta_i \frac{g_i}{\|g_i\|_2}\right\|_2^2.
\end{equation}
The following Theorem shows the main characteristic property of this direction.

\begin{theorem}
The direction $d_b$ defined in~\eqref{harmgrad:eq:d_b} satisfies the following equality
\[
(d_b, g_i) = \| d_b \|_2^2 \Vert g_i \Vert_2,
\]
for all $i$ such that $\beta^*_i \ne 0$. 
If $\beta^*_i \ne 0$ for all $i=1,\ldots,T$, the direction $d_b$ has the same angle with all individual gradients, i.e. this direction and gradients $g_i$ are equiangular with some angle $\alpha_i$.
\label{harmgrad::th::edm_claim}
\end{theorem}

\begin{proof}
Let $\mathcal{I} = \{ i_1, \ldots, i_K \}$ be the index set such that $i_p \in \mathcal{I}$ is equivalent to $\beta_{i_p}^* > 0$.
After that, the problem~\eqref{harmgrad:normgrad} can be re-written as
\[
\hat{\beta}^* = 
\argmin_{\hat{\beta} \in N_{K}} \left\| \sum\limits_{k=1}^{K} \hat{\beta}_k \frac{g_{i_k}}{\|g_{i_k}\|_2} \right\|_2^2, \quad N_{K} = 
\left\{\hat{\beta} \in \mathbb{R}^{K}_{++} \mid  \sum\limits_{k = 1}^{K} \hat{\beta}_{k} = 1\right\}.
\]
The solution $\hat{\beta}^*$ uniquely defines $\beta^*$ in the following way:
\begin{equation}
    \beta^*_i =
\begin{cases}
0, & \text{if } i \not\in \mathcal{I}\\
\hat{\beta}^*_k, & \text{if } i = i_k \in \mathcal{I}. 
\end{cases}
\label{harmgrad::eq::nnz_beta}
\end{equation}

Now we can write Lagrangian and derive KKT optimality conditions for this optimization problem.
Note that the constraint $\hat{\beta} \in \mathbb{R}^K_{++}$ is used in KKT conditions implicitly since by construction there exists solution that satisfies this constraint.
Therefore, the Lagrangian is the following function $L(\hat{\beta}, \lambda) = \left\| \sum\limits_{k=1}^{K} \hat{\beta}_k \frac{g_{i_k}}{\|g_{i_k}\|_2} \right\|_2^2 + \lambda \left(\sum\limits_{k = 1}^{K} \hat{\beta}_k - 1\right)$.
From the KKT optimality conditions follows that the gradient of $L$ with respect to $\hat{\beta}_j$ has to be zero in $\hat{\beta}^*$ and $\lambda^*$ for $j \in \mathcal{I}$: 
\[
\frac{\partial L(\hat{\beta}^*, \lambda^*)}{\partial \hat{\beta}_j} = 
2 \sum_{k = 1}^{K} \hat{\beta}^*_k\frac{(g_{i_k}, g_{i_j})}{\|g_{i_k}\|_2\|g_{i_j}\|_2} + \lambda^* = 
2 \left( \frac{g_{i_j}}{\|g_{i_j}\|_2}, \sum_{i=1}^T\beta^*_i \frac{g_{i}}{\|g_{i}\|_2}\right) + \lambda^* = 
2 \left(\frac{g_{i_j}}{\|g_{i_j}\|_2}, d_b  \right) + \lambda^* = 0,
\]
since only indices from the set $\mathcal{I}$ correspond to nonzero elements of $\beta^*$~\eqref{harmgrad::eq::nnz_beta}.
Thus we get the equality for any index $i=1,\ldots,T$ such that $\beta^*_i \neq 0$:
\begin{equation}
    (d_b, g_i) = -\frac{\lambda^*}{2} \| g_i\|_2.
    \label{harmgrad::eq::dbgi_lambda}
\end{equation}
Now we can compute the value of the remaining factor $-\frac{\lambda^*}{2}$.
Consider the following chain of equalities:
\[
\|d_b\|_2^2 = (d_b, d_b) = \left(d_b, \sum_{i=1}^T \beta_i^* \frac{g_i}{\|g_i\|_2}\right) = \sum_{i=1}^T \beta_i^* \frac{(d_b, g_i)}{\|g_i\|_2} = -\sum_{i=1}^T \beta_i^* \frac{\lambda^*}{2} = -\frac{\lambda^*}{2}.
\]
Thus, the equality~\eqref{harmgrad::eq::dbgi_lambda} can be re-written in the final form:
\[
(d_b, g_i) = \|d_b\|_2^2 \|g_i\|_2,
\]
where $i$ is any index such that $\beta_i^* > 0$.
\end{proof}

\subsection{Normalization of the equiangular direction $d_b$}
To define the direction $d_b$ only angles between gradients $g_i, i=1,\ldots, T$ are important.
But to define the proper norm of this direction, we need additional assumptions.
If we have only two gradients $g_1$ and $g_2$ and their norms are equal, then it is natural to require that the normalized vector $\gamma d_b$ coincides with the vector $d_h$, i.e., the vector $\gamma d_b$ has to belong to the convex hull of gradients $g_1$ and $g_2$.
To satisfy this requirement, the scale factor $\gamma$ and the corresponding vector $\gamma d_b$ have the following forms:
\begin{equation}
\gamma = \left(\sum\limits_{i=1}^T \frac{\beta^*_i}{\|g_i\|_2}\right)^{-1}  \quad \text{and} \quad 
    \gamma d_b = \left(\sum\limits_{i=1}^T \frac{\beta^*_i}{\|g_i\|_2}\right)^{-1} \sum_{i=1}^T \beta_i^* \frac{g_i}{\|g_i\|_2}.
    \label{harmgrad::eq::norm_db}
\end{equation}

\begin{remark}
For $T=2$ there is an explicit formula for the bisector direction:
\[
    d_b =  \frac{g_1}{\Vert g_1 \Vert} + \frac{g_2}{\Vert g_2 \Vert},
\]
that is formally not equal to the solution of the problem~\eqref{harmgrad:normgrad}, but only up to a normalization factor.
Now to get the normalized vector $\gamma d_b$ that belongs to the convex hull of the gradients $g_1, g_2$ we use the scale factor $\gamma = \left(\frac{1}{\|g_1\|_2} + \frac{1}{\|g_2\|_2}\right)^{-1}$ and obtain the final form of the  vector $\gamma d_b$:
\begin{equation}
\label{harmgrad:bisector2}
\gamma d_b = \frac{\left(\frac{g_1}{\Vert g_1 \Vert} + \frac{g_2}{\Vert g_2 \Vert} \right)}{\left(\frac{1}{\Vert g_1 \Vert} + \frac{1}{\Vert g_2 \Vert} \right)}.
\end{equation}
Note that if $\Vert d_b \Vert_2 > 0$, the direction defined by \eqref{harmgrad:bisector2} provides a guaranteed descent direction for both of the losses.
\end{remark}

\begin{theorem}
The equiangular direction method converges to the Pareto-stationary point $\hat{\theta}$ in finite number of iterations.
If the sequence $\{\theta_k\}$ is infinite, then there exists subsequence that converges to the Pareto-stationary point. 
\end{theorem}

\begin{proof}
If the EDM converges after $k^*$ iterations to the point $\theta_{k^*}$ such that $\|\gamma d_b\|_2 = 0$, then there exists $\beta^*$ such that $\gamma \sum\limits_{i=1}^T \beta^*_i \frac{g_i}{\|g_i\|_2} = 0$.
Therefore, $\theta_{k^*}$ is a Pareto-stationary point since $\sum\limits_{i=1}^T \alpha_i^* g_i = 0$, where $\alpha_i^* = \gamma \frac{\beta_i^*}{\|g_i\|_2}$ and $\alpha^* \in \Delta_T$.
On the other hand, if the sequence $\{\theta_k\}$ is infinite, then the proof is the same as the proof of Theorem 2.3. from~\cite{desideri2012multiple} on the convergence of MGDA. 
\end{proof}

The proposed method is summarized in Algorithm~\ref{harmgrad::alg::proposed}.

\begin{algorithm}[!htb]
\centering
\caption{Equiangular direction method}
\label{harmgrad:alg}
\begin{algorithmic}[1]
\REQUIRE{Losses $\mathcal{L}_1, \ldots, \mathcal{L}_T$, initial point $\theta_0$, number of iterations $I$, tolerance $\varepsilon$, learning rate $s$}
\ENSURE{Pareto-stationary point $\hat{\theta}$}
\FOR{$k=0,\ldots,I-1$}
\STATE{Compute individual gradients $g_i = \nabla \mathcal{L}_i(\theta_k)$, for all $i = 1,\ldots, T$}
\STATE{Solve optimization problem~\eqref{harmgrad:normgrad} with Algorithm~\ref{harmgrad::alg::fw} and obtain $\beta^*$}
\STATE{Compute direction $d^{(k)}_b$~\eqref{harmgrad:eq:d_b}}
\STATE{Normalize obtained direction $d^{(k)}_b$ according to equation~\eqref{harmgrad::eq::norm_db}, $\bar{d}_b^{(k)} = \gamma d^{(k)}_b$}
\IF{$\|\bar{d}^{(k)}_b\|_2 \leq \varepsilon$}
\STATE{$\hat{\theta} = \theta_k$}
\RETURN $\hat{\theta}$
\ENDIF
\STATE{Update: $\theta_{k+1} = \theta_k - s \cdot \bar{d}^{(k)}_b$}
\ENDFOR
\STATE{$\hat{\theta} = \theta_I$}
\RETURN $\hat{\theta}$
\end{algorithmic}
\label{harmgrad::alg::proposed}
\end{algorithm}

\subsection{Frank-Wolfe method to find $\beta^*$}
Since the feasible set in problem~\eqref{harmgrad:normgrad} is a simplex, this convex optimization problem can be efficiently solved by Frank-Wolfe method~\cite{frank1956algorithm,jaggi2013revisiting}.
Note that the the objective function in problem~\eqref{harmgrad:normgrad} can be written as
\begin{equation}
    \left\| \sum_{i=1}^T \beta_i u_i \right\|_2^2 =  \sum_{i=1}^T \beta_i^2 u^{\top}_iu_i + 2\sum_{i=1}^T\sum_{j=1}^{i-1} \beta_i\beta_j u_i^{\top}u_j = \beta^{\top}M\beta,
    \label{harmgrad::eq::fw_obj}
\end{equation}
where $u_i = \frac{g_i}{\|g_i\|_2}$ and $M$ is a matrix such that $M_{ij} = u_i^{\top}u_j$.
Then, one iteration of the Frank-Wolfe method for solving problem~~\eqref{harmgrad:normgrad} reduces to the following steps.
The first step is computing the gradient $2M\beta$ of the objective function~\eqref{harmgrad::eq::fw_obj} and find the index $i^*$ of its smallest element.
The second step is computing the optimal step size by solving auxiliary one-dimensional optimization problem:
\begin{equation}
    \eta^* = \argmin_{\eta \in [0, 1]} ((1 - \eta)\beta + \eta e_{i^*})^{\top}M((1 - \eta)\beta + \eta e_{i^*}),
    \label{harmgrad::gamma}
\end{equation}
where $e_{i^*}$ is the $i^*$-th basis vector.
The solution of problem~\eqref{harmgrad::gamma} can be written in the closed form:
\begin{equation}
    \eta^* =
\begin{cases}
0, & \text{if } \beta^{\top}M\beta \leq \beta^{\top}Me_{i^*},\\
1, & \text{if } e_{i^*}^{\top}Me_{i^*} \leq \beta^{\top}Me_{i^*}, \\
\frac{\beta^{\top}M(\beta - e_{i^*})}{(e_{i^*} - \beta)^{\top}M(e_{i^*} - \beta)}, & \text{otherwise.}
\end{cases}
\label{harmgrad::eq::gamma_solution}
\end{equation}
The third step is updating coefficients of convex combination as
\[
\beta^{(k+1)} = (1 - \eta^*)\beta^{(k)} + \eta^* e_{i^*}.
\]

For convenience we provide the detailed description of the Frank-Wolfe method to solve the problem~\eqref{harmgrad:normgrad} in Algorithm~\ref{harmgrad::alg::fw}. 
Note that every iteration of the Frank-Wolfe method that solves problem~\eqref{harmgrad:normgrad} can be interpreted from the geometric perspective.
In particular, the angle between the basis vector $e_{i^*}$ and gradient $2M\beta^{(k)}$ is maximum among angles between the gradient and the basis vectors, see line 4 in Algorithm~\ref{harmgrad::alg::fw}.
 
\begin{algorithm}[!h]
\caption{Frank-Wolfe method to solve~\eqref{harmgrad:normgrad}}
\begin{algorithmic}[1]
\REQUIRE{Normalized gradients $u_i = \frac{g_i}{\|g_i\|_2}, i = 1,\ldots, T$, tolerance $\varepsilon$, maximum number of iterations $K$}
\ENSURE{Coefficients $\beta^*$ of the convex combination of $u_i$, whose norm is minimum}
\STATE{$\beta^{(0)} = \frac{1}{T}\mathbf{1}_T$, where $\mathbf{1}_T$ is $T$-dimensional vector with all ones}
\STATE{Precompute a matrix $M \in \mathbb{R}^{T \times T}$ such that $M_{ij} = u_i^{\top}u_j$} 
\FOR{$k=0, \ldots, K-1$}
\STATE{$i^* = \argmin\limits_{i \in \{1,\ldots,T\}} \sum\limits_{j=1}^T M_{ij}\beta^{(k)}_j$}
\STATE{$\eta^* = \argmin\limits_{\eta \in [ 0, 1]} \; ((1 - \eta) \beta^{(k)} + \eta e_{i^*})^{\top} M ((1 - \eta) \beta^{(k)} + \eta e_{i^*})$, where $e_{i^*}$ is the $i^*$-th basis vector. The solution is computed with formula~\eqref{harmgrad::eq::gamma_solution}}
\IF{$\eta^* \leq \varepsilon$}
\STATE{$\beta^* = \beta^{(k)}$}
\RETURN $\beta^*$
\ENDIF
\STATE{$\beta^{(k+1)} = (1 - \eta^*)\beta^{(k)} + \eta^* e_{i^*}$} 
\ENDFOR
\STATE{$\beta^* = \beta^{(K)}$}
\RETURN{$\beta^*$}
\end{algorithmic}
\label{harmgrad::alg::fw}
\end{algorithm}

\section{Computational experiment}
In this section, we compare the proposed method with other approaches to solve imbalanced classification problem and multi-task learning problem.
As an example of the latter problem, we consider classification problem on the MultiMNIST dataset, which is a modification of the standard MNIST dataset~\cite{lecun1998gradient} appropriate for multi-task learning.
To compute the gradients of losses in the considered problems, we use automatic differentiation technique implemented in PyTorch framework~\cite{paszke2017automatic}.
The source code can be found at GitHub\footnote{\url{https://github.com/amkatrutsa/edm}}.

\subsection{Imbalanced classification problem}

Let $(x_j, y_j), j=1, \ldots, N$ be the dataset with $c$ classes, i.e. $y_j \in Y = \{0, \ldots, c-1 \}, j = 1, \ldots, N$.
Denote by $X$ a set of samples $x_j \in \mathbb{R}^m$.
Consider a function $f: X \times \mathbb{R}^d \to Y$ that estimates label $y \in Y$ of a sample $x \in X$.
We need to find a parameter $\theta^* \in \mathbb{R}^d$ such that the classification quality will be as high as possible.
To measure classification quality, the loss function $\mathcal{L}$ is introduced and is minimized with respect to $\theta$.
This loss function is typically written as
\begin{equation}
\label{harmgrad:multi-class}
    \mathcal{L}(\theta) = \sum_{i=0}^{c-1} \mathcal{L}_i(\theta),
\end{equation}
where $\mathcal{L}_i$ is the loss corresponding to the $i$-th class:
\begin{equation*}
    \mathcal{L}_i(\theta) = \sum_{j \in C_i} l(\theta \mid (x_j, y_j)), 
\end{equation*}
where $C_i, \; i = 0, \ldots, c- 1$ is a set of indices such that if $k \in C_i$ then $y_k = i$ and $l$ is a loss for a given pair $(x_j, y_j)$.
We use cross-entropy loss function 
and represent $f$ as a neural network.
Therefore, the vector~$\theta$ is composed of the stacked vectorized parameters of this neural network.

The classification problem is called imbalanced if there exists class label $y^* \in Y$ such that $|C_{y^*}| \ll |C_y|$, for all $y \in Y \setminus y^*$.
In other words, the number of samples from the class $y^*$ is significantly smaller than the number of samples from the other classes.
In the case of binary imbalanced classification, where $c=2$, the class $y^*$ is called \emph{minor} and the other class is called \emph{major}.
Further, we always refer the label of the major class as $0$ and the label of the minor class as $1$.
If one always assigns to any sample the label $0$, the standard accuracy computed over all samples will be close to~1.
However, it means that the samples with ground-truth label $1$ are always misclassified.
To address this issue, a class weight is introduced to balance $\mathcal{L}_0$ and $\mathcal{L}_1$.
Denote by $\mu \geq 1$ the weight corresponding to the minor class.
Then the total loss function can be re-written in the form
\begin{equation}
    \mathcal{L}(\theta) = \mathcal{L}_0(\theta) + \mu\mathcal{L}_1(\theta).
    \label{harmgrad:eq:imbalanced_class_weight_loss}
\end{equation}
The higher the value of $\mu$, the higher classification accuracy in the minor class is expected.  
However, the hyperparameter $\mu$ has to be tuned to balance accuracies of the major and minor classes.
To avoid this tuning, EDM can be used to automatically balance $\mathcal{L}_0$ and $\mathcal{L}_1$ without introducing additional hyperparameter.

To compare considered methods in the imbalanced classification problem, we use the credit card transaction dataset~\cite{dal2017credit}, where the minor class consists of fraud transactions.
The number of samples in this dataset is $284807$, and the number of features is $30$.
Note that number of fraud transactions is only $492$, which is $0.17\%$ of the total number of samples.
Thus, we have an imbalanced binary classification problem.

To demonstrate how EDM adjusts accuracies in the imbalanced classification problem, we compare EDM with the vanilla SGD method that minimizes the total loss $\mathcal{L}$~\eqref{harmgrad:eq:imbalanced_class_weight_loss} for $\mu = 1$ and $\mu= 10$.
More advanced gradient-based methods still suffer from the necessity of tuning hyperparameter $\mu$.
Therefore, we compare EDM with the vanilla SGD method only.
We consider the simple neural network with two fully-connected layers, ReLU activation between them and hidden dimension equal to $100$.
The entire dataset is split in train and test sets such that the portion of the minor class is the same.
The numbers of samples in the train and test sets are 227845 and 56962, respectively.
Since the classes are imbalanced, we generate batches for them separately.
The batch sizes are different, but during every epoch 40 batches of every class are used in the training process.
We test different learning rates $s \in \{10^{-1}, 10^{-2}, 10^{-3} \}$ in considered methods.
Since the smaller learning rate induces the larger number of epochs, we use different numbers of epochs in training. 
In particular, learning rates $s = 10^{-3}$ and $s = 10^{-2}$ in all considered methods require $30$ epochs for the convergence.
In the case of learning rate $10^{-3}$, SGD with $\mu=1$ and $\mu=10$ converges after $30$ epochs, but EDM and MGDA require $150$ and $300$ epochs for convergence, respectively.

Table~\ref{harmgrad::tab::imbalanced_class} presents the test accuracies separately for the minor and major classes.
It shows that EDM gives a reasonable trade-off between accuracies on such imbalanced classes without any tuning of hyperparameter.
Moreover, EDM is robust to a range of step sizes and preserves balanced accuracies for individual classes. 
Also, EDM gives higher accuracy for the major class and slightly smaller or equal accuracy for the minor class compare with MGDA, see Tables~\ref{harmgrad::tab::imbalanced_class_s10-3} and~\ref{harmgrad::tab::imbalanced_class_s10-2}.
In the case of using learning rate $s = 10^{-1}$, EDM provides higher accuracies for both major and minor classes than MGDA.

\begin{table}[!h]
    \centering
    \caption{Test accuracies for both classes given by the considered methods. 
    The reported mean values and standard deviations are computed from three random initializations of the considered model.}
    \begin{subtable}{\textwidth}
    \centering
    \begin{tabular}{ccccc}
    Class & EDM & MGDA & SGD, $\mu=1$ & SGD, $\mu = 10$ \\
    \toprule
    Minor &  $0.918 $ & $0.925 \pm 5\cdot 10^{-3}$ & $0.895 \pm 5 \cdot 10^{-3}$ & $0.966 \pm 5 \cdot 10^{-3}$  \\
    Major & $0.953 \pm 10^{-3}$ & $0.925 \pm 7\cdot 10^{-3}$ & $0.9843 \pm 5\cdot 10^{-4}$ & $0.837 \pm 4 \cdot 10^{-3}$\\
\end{tabular}
\caption{Learning rate $s = 10^{-3}$}
\label{harmgrad::tab::imbalanced_class_s10-3}
\end{subtable}
\\
    \begin{subtable}{\textwidth}
    \centering
    \begin{tabular}{ccccc}
    Class & EDM & MGDA & SGD, $\mu=1$ & SGD, $\mu = 10$ \\
    \toprule
    Minor & $0.918$ & $0.918$ & $0.901 \pm 5 \cdot 10^{-3}$ & $0.966 \pm 5 \cdot 10^{-3}$ \\
    Major & $0.954 \pm 2 \cdot 10^{-3}$ & $0.949 \pm 2\cdot 10^{-3}$ & $0.984 \pm 2\cdot 10^{-3}$ & $0.917 \pm 10^{-3}$\\
\end{tabular}
\caption{Learning rate $s = 10^{-2}$}
\label{harmgrad::tab::imbalanced_class_s10-2}
\end{subtable}
\\
\begin{subtable}{\textwidth}
\centering
    \begin{tabular}{ccccc}
    Class & EDM & MGDA & SGD, $\mu=1$ & SGD, $\mu = 10$ \\
    \toprule
    Minor & $0.904 \pm 9 \cdot 10^{-3}$ & $0.901 \pm 5\cdot 10^{-3}$ & $0.881 \pm 5 \cdot 10^{-3}$ & $0.894 \pm 4 \cdot 10^{-3}$  \\
    Major & $0.982 \pm 10^{-3}$ & $0.983 \pm 2\cdot 10^{-3}$  & $0.9924 \pm 8 \cdot 10^{-4}$  & $0.98584 \pm  4 \cdot 10^{-4}$\\
\end{tabular}
\caption{Learning rate $s=10^{-1}$}
\label{harmgrad::tab::imbalanced_class_s10-1}
\end{subtable}
    \label{harmgrad::tab::imbalanced_class}
\end{table}

\subsection{Multi-task learning problem}
In this section, the standard classification problem is reduced to the multi-task learning (MTL) problem following the study~\cite{sener2018multi}.
To test the presented method in solving the MTL problem, we consider the MultiMNIST dataset and adaptation of the LeNet neural network~\cite{lecun1989backpropagation}.
MultiMMIST dataset is a modification of the classical MNIST dataset~\cite{lecun1998gradient}.
Every image from the MultiMNIST is composed of two MNIST images: one image is placed in the top-left corner, and the other one is placed in the bottom-right corner~\cite{sabour2017dynamic,sener2018multi}. 
To make overlaying consistent, we create an image of size $32 \times 32$, place digits from the original MNIST dataset to the opposite corners and finally scale it to the standard size $28 \times 28$.
Samples from MultiMNIST are shown in Figure~\ref{harmgrad::fig::multimnist_example}.
\begin{figure}[!h]
    \centering
    \includegraphics[width=\textwidth]{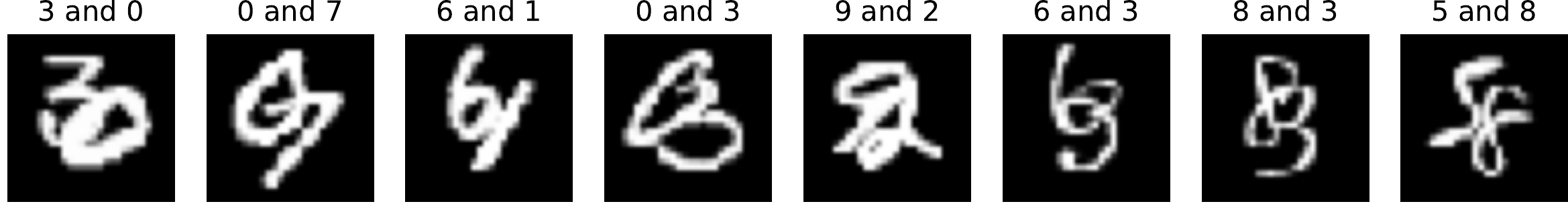}
    ~
    \caption{Samples from MultiMNIST dataset. 
    Every sample is composed by taking two images with different labels from MNIST and placing one image to the top-left and another image to the bottom-right}
    \label{harmgrad::fig::multimnist_example}
\end{figure}

Now we have the MTL problem, where the first task is to classify an image in the top-left corner, and the second task is to classify an image in the bottom-right corner.
To solve this MTL problem, the following modification of the LeNet architecture~\cite{lecun1989backpropagation} is presented in~\cite{sener2018multi} and is used in this study, see Figure~\ref{harmgrad::fig::lenet}.
Shared layers generate a representation of every image that is used to solve both tasks.
Task-specific layers are responsible for solving a particular task, use representation, constructed by the shared layers, and do not affect the solution of the other task.
\begin{figure}[!h]
    \centering
    \includegraphics[width=0.7\textwidth]{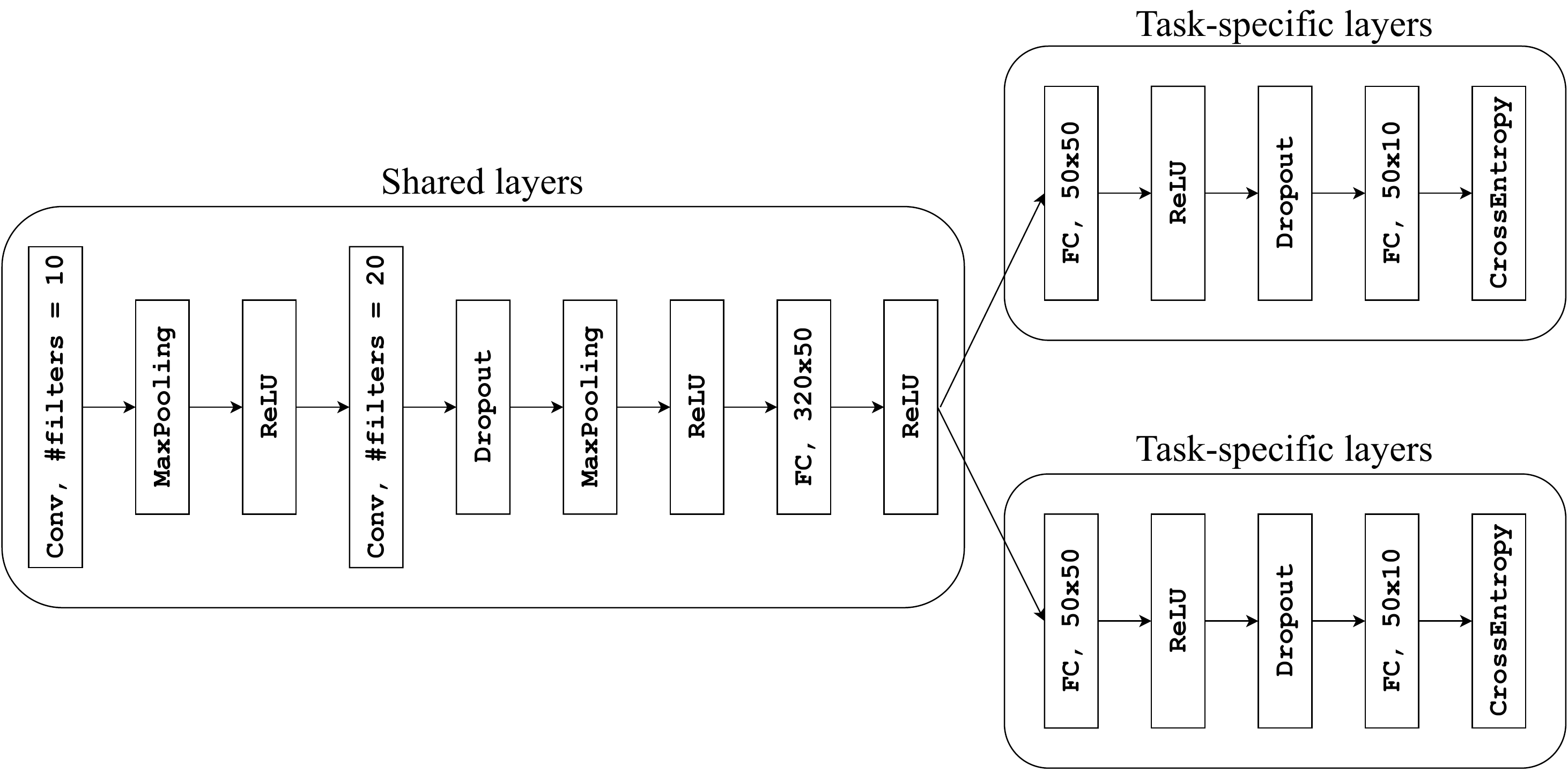}
    \caption{Modified LeNet neural network to solve multi-task learning problem. 
    The kernel size in \texttt{Conv} layers is $5$. 
    The kernel size in \texttt{MaxPooling} layers is 2. 
    Probability to zero an element in \texttt{Dropout} layers is~$0.2$. 
    Denote by \texttt{FC, $\mathtt{d_{in}}$x$\mathtt{d_{out}}$} fully-connected layer, where the input dimension is $\mathtt{d_{in}}$ and the output dimension is $\mathtt{d_{out}}$.}
    \label{harmgrad::fig::lenet}
\end{figure}
Following the work~\cite{sener2018multi}, we update parameters in shared and task-specific layers differently.
The parameters of the shared layers are updated based on the multi-objective optimization methods since these parameters affect losses corresponding to both tasks simultaneously.
The parameters of the task-specific layers are updated with the SGD method since these parameters affect only single-task losses.

In the presented experiments, we illustrate the robustness of the EDM to the multi-scale losses in contrast to the MGDA method.
Denote by $\mathcal{L}_1$ and $\mathcal{L}_2$ cross-entropy losses for the first and the second tasks, respectively. 
To control the scale of the loss $\mathcal{L}_2$, we introduce a hyper-parameter $\kappa \geq 1$ that is multiplied by the loss $\mathcal{L}_2$. 
In this setting we minimize the losses $\mathcal{L}_1$ and $\hat{\mathcal{L}}_2 = \kappa \mathcal{L}_2$ concurrently.
The larger $\kappa$, the larger loss $\hat{\mathcal{L}}_2$.
In particular, if $\kappa = 1$, then losses $\mathcal{L}_1$ and $\mathcal{L}_2$ are minimized as they are. 
Note that the magnitudes of both losses are approximately the same for the considered dataset and neural network.
At the same time, if $\kappa = 10$, then the loss $\hat{\mathcal{L}}_2$ becomes ten times larger.
We expect that in this case, EDM ensures more robust training than MGDA and, consequently, higher test accuracy.
Also, we compare EDM with the single-task approach~\cite{sener2018multi}.
This approach is based on two identical task-specific neural networks such that every neural network solves a particular task, i.e., classification of the top-left or the bottom-right image.
The architecture of these neural networks coincide with the LeNet modification in Figure~\ref{harmgrad::fig::lenet}, but including only one block of task-specific layers since every network solves only one task. 
The single-task neural networks are trained with the vanilla SGD method for a fair comparison with EDM and MGDA.

Table~\ref{harmgrad::tab::multimnist} presents the test accuracy obtained by the considered methods for $\kappa=1$ and $\kappa=50$. 
In this experiment we use learning rate $s = 0.05$, batch size $256$ and $25$ epochs.
We show in Table~\ref{harmgrad::tab::multimnist_kappa1} that even for $\kappa=1$ EDM is more or equally accurate in both tasks compared with MGDA and the single-task approach.
The desired property of the multi-objective optimization method to be robust to different scales of individual losses is more clear from Table~\ref{harmgrad::tab::multimnist_kappa50}.
This table corresponds to the setting, where loss $\mathcal{L}_2$ is multiplied by the factor $\kappa = 50$.
Test accuracy in both classes given by EDM are significantly higher than test accuracy corresponding to MGDA and single teask approach.
Naturally, the single-task approach gives the same test accuracies for the top-left class for both values of $\kappa$, since the corresponding loss $\mathcal{L}_1$ is unchanged.


\begin{table}[!h]
    \centering
    \caption{Test accuracies given by the considered methods for both classes. 
    The reported mean values and standard deviations are computed from three random initializations of the considered model.}
    \begin{subtable}{\textwidth}
    \centering
    \begin{tabular}{cccc}
    Class & EDM & MGDA & Single task \\
    \toprule
    Top-left &  $0.9562 \pm 0.0011$ & $\mathbf{0.9565 \pm 0.0005}$ & $0.9482 \pm 0.0021$  \\
    Bottom-right & $\mathbf{0.9413 \pm 0.0011}$ & $0.9400 \pm 0.0004$ & $0.9306 \pm 0.0021$\\
\end{tabular}
\caption{$\kappa = 1$}
\label{harmgrad::tab::multimnist_kappa1}
\end{subtable}
    \begin{subtable}{\textwidth}
    \centering
    \begin{tabular}{cccc}
    Class & EDM & MGDA & Single task \\
    \toprule
    Top-left & $\bf{0.9552 \pm 0.0050}$ & $0.1032  \pm 0.0073 $ & $0.9482 \pm 0.0021$ \\
    Bottom-right & $\bf{0.8639 \pm 0.0801}$ & $0.1027 \pm 0.0034 $ & $0.1075 \pm 0.0000$\\
\end{tabular}
\caption{$\kappa = 50$}
\label{harmgrad::tab::multimnist_kappa50}
\end{subtable}
    \label{harmgrad::tab::multimnist}
\end{table}



\section{Related works}
Multi-objective optimization problems come from many applications, where the target vector is evaluated by more than one loss function.
Examples of such applications are computer networks~\cite{donoso2016multi}, energy saving~\cite{cui2017multi}, engineering applications~\cite{chiandussi2012comparison,marler2004survey}, etc.
In machine learning and deep learning, models are typically trained by minimizing some pre-defined loss function computed on the training dataset.
However, the quality of the trained model is additionally evaluated according to external criteria.
For example, the coefficient of determination in regression problems estimates the ratio of the dependent variable variance that is explained by the trained model~\cite{helland1987interpretation}. 
In classification problems, AUC score close to one indicates the high quality of the trained  model~\cite{vanderlooy2008critical}.
In deep learning, neural networks for image classification can be evaluated on the robustness against the adversarial attacks~\cite{yuan2019adversarial,tursynbek2020geometry}.
Also, recently proposed neural ODE models can be compared not only based on the primal quality measure but also based on the smoothness of the trained dynamic~\cite{gusak2020towards}.
Although, most of the external criteria to evaluate the quality of models depend on the discrete variables, gradient-free methods to solve multi-objective discontinuous optimization problems are proposed~\cite{deb2001multi}.
One of the most common approaches are genetic algorithms~\cite{deb2002fast,von2014survey,abraham2005evolutionary}, particle swarm optimization methods~\cite{wang2009particle} and other nature-inspired heuristics~\cite{coello2009advances,omkar2011artificial}. 
These methods randomly explore the search space to find Pareto-optimal points and mostly suffer from the absence of any guarantees on the convergence to the Pareto-optimal point.

We focus on the unconstrained multi-objective optimization problems, where the loss functions are differentiable.
Methods to solve such problems can use gradients of individual losses.
Besides the weighting sum method that is typically used in applications~\cite{liu2016gradient}, the modifications of the standard methods for single-objective optimization problems are proposed.  
For example, the steepest descent method for multi-objective optimization problems is proposed in~\cite{fliege2000steepest}.
This method requires solving the auxiliary optimization problem in every iteration to get descent direction, which is similar to MGDA.
Extension of this approach to the problems with box constraints is presented in~\cite{miglierina2008box}.
Also, there is given the interpretation of multi-objective optimization problems from the dynamical system theory perspective.
Further, the proximal method to solve multi-objective optimization problems is proposed in~\cite{bonnel2005proximal}, but without any numerical comparison with other methods.
One more approach to solving multi-objective problems is the generalized homotopy approach~\cite{hillermeier2001generalized}.
It represents Pareto-optimal points as a differentiable manifold and generates new Pareto-optimal points through numerical evaluation of a local chart of this manifold.

%

\section{Conclusion}
This study considers multi-objective optimization problems, where loss functions are of different scales. 
To solve problems with such property, we  propose the Equiangular Direction Method (EDM) and proof that it guarantees equal relative decrease of every loss function.
Thus, EDM is robust to multi-scale losses.
We illustrate the performance of the EDM in solving the imbalanced classification and multi-task learning problems.
The proposed method provides the highest test accuracy compared with other approaches to solve considered problems.

\bibliographystyle{siamplain}
\bibliography{main}

\begin{thebibliography}{10}

\bibitem{abraham2005evolutionary}
{\sc A.~Abraham and L.~Jain}, {\em Evolutionary multiobjective optimization},
  in Evolutionary Multiobjective Optimization, Springer, 2005, pp.~1--6.

\bibitem{bonnel2005proximal}
{\sc H.~Bonnel, A.~N. Iusem, and B.~F. Svaiter}, {\em Proximal methods in
  vector optimization}, SIAM Journal on Optimization, 15 (2005), pp.~953--970.

\bibitem{chiandussi2012comparison}
{\sc G.~Chiandussi, M.~Codegone, S.~Ferrero, and F.~E. Varesio}, {\em
  Comparison of multi-objective optimization methodologies for engineering
  applications}, Computers \& Mathematics with Applications, 63 (2012),
  pp.~912--942.

\bibitem{coello2009advances}
{\sc C.~C. Coello, C.~Dhaenens, and L.~Jourdan}, {\em Advances in
  multi-objective nature inspired computing}, vol.~272, Springer, 2009.

\bibitem{cui2017multi}
{\sc Y.~Cui, Z.~Geng, Q.~Zhu, and Y.~Han}, {\em Multi-objective optimization
  methods and application in energy saving}, Energy, 125 (2017), pp.~681--704.

\bibitem{dal2017credit}
{\sc A.~Dal~Pozzolo, G.~Boracchi, O.~Caelen, C.~Alippi, and G.~Bontempi}, {\em
  Credit card fraud detection: a realistic modeling and a novel learning
  strategy}, {IEEE transactions on neural networks and learning systems}, 29
  (2017), pp.~3784--3797.

\bibitem{das1997closer}
{\sc I.~Das and J.~E. Dennis}, {\em A closer look at drawbacks of minimizing
  weighted sums of objectives for pareto set generation in multicriteria
  optimization problems}, Structural optimization, 14 (1997), pp.~63--69.

\bibitem{deb2001multi}
{\sc K.~Deb}, {\em Multi-objective optimization using evolutionary algorithms},
  vol.~16, John Wiley \& Sons, 2001.

\bibitem{deb2002fast}
{\sc K.~Deb, A.~Pratap, S.~Agarwal, and T.~Meyarivan}, {\em A fast and elitist
  multiobjective genetic algorithm: Nsga-ii}, IEEE transactions on evolutionary
  computation, 6 (2002), pp.~182--197.

\bibitem{desideri2012multiple}
{\sc J.-A. D{\'e}sid{\'e}ri}, {\em {Multiple-gradient descent algorithm (MGDA)
  for multiobjective optimization}}, Comptes Rendus Mathematique, 350 (2012),
  pp.~313--318.

\bibitem{donoso2016multi}
{\sc Y.~Donoso and R.~Fabregat}, {\em Multi-objective optimization in computer
  networks using metaheuristics}, CRC Press, 2016.

\bibitem{fliege2000steepest}
{\sc J.~Fliege and B.~F. Svaiter}, {\em Steepest descent methods for
  multicriteria optimization}, Mathematical Methods of Operations Research, 51
  (2000), pp.~479--494.

\bibitem{frank1956algorithm}
{\sc M.~Frank, P.~Wolfe, et~al.}, {\em An algorithm for quadratic programming},
  Naval research logistics quarterly, 3 (1956), pp.~95--110.

\bibitem{ghane2015new}
{\sc A.~Ghane-Kanafi and E.~Khorram}, {\em A new scalarization method for
  finding the efficient frontier in non-convex multi-objective problems},
  Applied Mathematical Modelling, 39 (2015), pp.~7483--7498.

\bibitem{gusak2020towards}
{\sc J.~Gusak, L.~Markeeva, T.~Daulbaev, A.~Katrutsa, A.~Cichocki, and
  I.~Oseledets}, {\em Towards understanding normalization in neural odes},
  arXiv preprint arXiv:2004.09222,  (2020).

\bibitem{helland1987interpretation}
{\sc I.~S. Helland}, {\em On the interpretation and use of r2 in regression
  analysis}, Biometrics,  (1987), pp.~61--69.

\bibitem{hillermeier2001generalized}
{\sc C.~Hillermeier}, {\em Generalized homotopy approach to multiobjective
  optimization}, Journal of Optimization Theory and Applications, 110 (2001),
  pp.~557--583.

\bibitem{jaggi2013revisiting}
{\sc M.~Jaggi}, {\em {Revisiting Frank-Wolfe: Projection-free sparse convex
  optimization.}}, in Proceedings of the 30th international conference on
  machine learning, 2013, pp.~427--435.

\bibitem{jin2006multi}
{\sc Y.~Jin}, {\em Multi-objective machine learning}, vol.~16, Springer Science
  \& Business Media, 2006.

\bibitem{katrutsa2015stress}
{\sc A.~Katrutsa and V.~Strijov}, {\em Stress test procedure for feature
  selection algorithms}, Chemometrics and Intelligent Laboratory Systems, 142
  (2015), pp.~172--183.

\bibitem{lecun1989backpropagation}
{\sc Y.~LeCun, B.~Boser, J.~S. Denker, D.~Henderson, R.~E. Howard, W.~Hubbard,
  and L.~D. Jackel}, {\em Backpropagation applied to handwritten zip code
  recognition}, Neural computation, 1 (1989), pp.~541--551.

\bibitem{lecun1998gradient}
{\sc Y.~LeCun, L.~Bottou, Y.~Bengio, and P.~Haffner}, {\em Gradient-based
  learning applied to document recognition}, Proceedings of the IEEE, 86
  (1998), pp.~2278--2324.

\bibitem{liu2014multiobjective}
{\sc C.~Liu, X.~Xu, and D.~Hu}, {\em Multiobjective reinforcement learning: A
  comprehensive overview}, IEEE Transactions on Systems, Man, and Cybernetics:
  Systems, 45 (2014), pp.~385--398.

\bibitem{liu2016gradient}
{\sc X.~Liu and A.~C. Reynolds}, {\em Gradient-based multi-objective
  optimization with applications to waterflooding optimization}, Computational
  Geosciences, 20 (2016), pp.~677--693.

\bibitem{marler2004survey}
{\sc R.~T. Marler and J.~S. Arora}, {\em Survey of multi-objective optimization
  methods for engineering}, Structural and multidisciplinary optimization, 26
  (2004), pp.~369--395.

\bibitem{miettinen2002scalarizing}
{\sc K.~Miettinen and M.~M. M{\"a}kel{\"a}}, {\em On scalarizing functions in
  multiobjective optimization}, OR spectrum, 24 (2002), pp.~193--213.

\bibitem{miglierina2008box}
{\sc E.~Miglierina, E.~Molho, and M.~C. Recchioni}, {\em Box-constrained
  multi-objective optimization: a gradient-like method without “a priori”
  scalarization}, European Journal of Operational Research, 188 (2008),
  pp.~662--682.

\bibitem{omkar2011artificial}
{\sc S.~Omkar, J.~Senthilnath, R.~Khandelwal, G.~N. Naik, and
  S.~Gopalakrishnan}, {\em Artificial bee colony (abc) for multi-objective
  design optimization of composite structures}, Applied Soft Computing, 11
  (2011), pp.~489--499.

\bibitem{paszke2017automatic}
{\sc A.~Paszke, S.~Gross, S.~Chintala, G.~Chanan, E.~Yang, Z.~DeVito, Z.~Lin,
  A.~Desmaison, L.~Antiga, and A.~Lerer}, {\em Automatic differentiation in
  pytorch},  (2017).

\bibitem{sabour2017dynamic}
{\sc S.~Sabour, N.~Frosst, and G.~E. Hinton}, {\em Dynamic routing between
  capsules}, in Advances in neural information processing systems, 2017,
  pp.~3856--3866.

\bibitem{sener2018multi}
{\sc O.~Sener and V.~Koltun}, {\em Multi-task learning as multi-objective
  optimization}, in Advances in Neural Information Processing Systems, 2018,
  pp.~527--538.

\bibitem{soda2011multi}
{\sc P.~Soda}, {\em A multi-objective optimisation approach for class imbalance
  learning}, Pattern Recognition, 44 (2011), pp.~1801--1810.

\bibitem{tursynbek2020geometry}
{\sc N.~Tursynbek, A.~Petiushko, and I.~Oseledets}, {\em Geometry-inspired
  top-k adversarial perturbations}, arXiv preprint arXiv:2006.15669,  (2020).

\bibitem{van2014multi}
{\sc K.~Van~Moffaert and A.~Now{\'e}}, {\em Multi-objective reinforcement
  learning using sets of pareto dominating policies}, The Journal of Machine
  Learning Research, 15 (2014), pp.~3483--3512.

\bibitem{vanderlooy2008critical}
{\sc S.~Vanderlooy and E.~H{\"u}llermeier}, {\em A critical analysis of
  variants of the auc}, Machine Learning, 72 (2008), pp.~247--262.

\bibitem{von2014survey}
{\sc C.~Von~L{\"u}cken, B.~Bar{\'a}n, and C.~Brizuela}, {\em A survey on
  multi-objective evolutionary algorithms for many-objective problems},
  Computational optimization and applications, 58 (2014), pp.~707--756.

\bibitem{wang2009particle}
{\sc Y.~Wang and Y.~Yang}, {\em Particle swarm optimization with preference
  order ranking for multi-objective optimization}, Information Sciences, 179
  (2009), pp.~1944--1959.

\bibitem{xue2012particle}
{\sc B.~Xue, M.~Zhang, and W.~N. Browne}, {\em Particle swarm optimization for
  feature selection in classification: A multi-objective approach}, IEEE
  transactions on cybernetics, 43 (2012), pp.~1656--1671.

\bibitem{yuan2019adversarial}
{\sc X.~Yuan, P.~He, Q.~Zhu, and X.~Li}, {\em Adversarial examples: Attacks and
  defenses for deep learning}, IEEE transactions on neural networks and
  learning systems, 30 (2019), pp.~2805--2824.

\bibitem{zhang2017survey}
{\sc Y.~Zhang and Q.~Yang}, {\em A survey on multi-task learning}, arXiv
  preprint arXiv:1707.08114,  (2017).

\end{thebibliography}
\end{document}